\documentclass[12pt]{amsart}
\usepackage{amsmath, amssymb}
\usepackage{enumerate} 
\makeatletter
\@namedef{subjclassname@2020}{\textup{2020} Mathematics Subject Classification}
\makeatother
\newtheorem{theorem}{\bf Theorem}[section]
\newtheorem{lemma}[theorem]{\bf Lemma}
\newtheorem{proposition}[theorem]{\bf Proposition}
\newtheorem{corollary}[theorem]{\bf Corollary}
\newtheorem{remark}[theorem]{\sc Remark}
\newtheorem{hypothesis}[theorem]{\bf Hypothesis}

\newtheorem{example}[theorem]{\bf Example}

\newcommand{\N}{\mathbb{N}}

\newcommand{\pr }{\mathrm{Pr} }

\newcommand{\ep}{\epsilon}

\DeclareMathOperator{\D}{E}

\newcommand{\Q }{Q }

\newcommand{\B }{B }

\begin{document}
\title[Commuting probability]{Commuting probability for conjugate subgroups of a finite group}
\thanks{The first and the third authors are members of GNSAGA (INDAM). The second author was partially supported by NSF grant DMS 1901595 and Simons Foundation Fellowship 609771.
The fourth author was  supported by  FAPDF and CNPq.
}

\author[E. Detomi]{Eloisa Detomi}
\address{Dipartimento di Matematica \lq\lq Tullio Levi-Civita\rq\rq, Universit\`a degli Studi di Padova, Via Trieste 63, 35121 Padova, Italy} 
\email{eloisa.detomi@unipd.it}
\author[R. M. Guralnick]{Robert M. Guralnick}
\address{ Department of Mathematics, University of
Southern California, Los Angeles, CA90089-2532, USA}
\email{guralnic@usc.edu}

\author[M. Morigi]{Marta Morigi}
\address{Dipartimento di Matematica, Universit\`a di Bologna\\
Piazza di Porta San Donato 5 \\ 40126 Bologna \\ Italy}
\email{marta.morigi@unibo.it}
\author[P. Shumyatsky]{Pavel Shumyatsky}
\address{Department of Mathematics, University of Brasilia\\
Brasilia-DF \\ 70910-900 Brazil}
\email{pavel@unb.br}

\subjclass[2020]{20D20; 20E45; 20P05} 
% 20D20 Sylow subgroups, Sylow properties, π-groups, π-structure 
%20E45 conjugacy classes, 
%20P05 Probabilistic methods in group theory da [3]  20N99 None of the above, but in this section, in 20Nxx Other generalizations of groups, 
\keywords{Commuting probability, Sylow subgroups, Simple groups}

\begin{abstract} 
Given two subgroups $H,K$ of a finite group $G$, the probability that a pair of random elements from $H$ and $K$ commutes is denoted by $\pr(H,K)$.
We address the following question.

Let $P$ be a $p$-subgroup of a finite group $G$ and assume that $\pr(P,P^x)\geq\ep>0$ for every $x\in G$. Is the order of $P$ modulo $O_p(G)$ bounded in terms of $\ep$ only?

With respect to this question, we establish several positive results but show that in general the answer is negative. In particular, we prove that if the composition factors of $G$, which are isomorphic to simple groups of Lie type in characteristic $p$, have Lie rank at most $n$, then the order of $P$ modulo $O_p(G)$ is bounded in terms of $n$ and $\ep$ only (Theorem \ref{main2}). If $P$ is a Sylow $p$-subgroup of $G$, then the order of $P$ modulo $O_p(G)$ is bounded in terms $\ep$ only (Theorem \ref{main1}). Some other results of similar flavour are established.

We also show that if $\pr(P_1,P_2)>0$ for every two Sylow $p$-subgroups $P_1,P_2$ of a profinite group $G$, then $O_{p,p'}(G)$ is open in $G$ (Theorem \ref{main4}).

 \end{abstract}
 \maketitle

\section{Introduction} 
The commuting probability $\pr(G)$ of a finite group $G$ is the probability that two random elements of $G$ commute. More generally,
if $H$ and $K$ are two subgroups of $G$, we write $\pr(H,K)$ for the probability that two randomly chosen elements from $H$ and $K$ commute. Thus, 
\[ \pr(H,K) =\frac{ | \{ (x,y) \in H\times K \mid xy=yx \} |}{|H|\,|K|}.\]

It is well-known that $\pr(G)\leq5/8$ for every nonabelian group G. Another important result is a
theorem of Peter Neumann which states that if $\pr(G)\geq\ep>0$, then $G$ has a subgroup $H$ such that the index $|G : H|$ and the order of $H'$ are both bounded in terms of $\ep$ only \cite{neumann}. In what follows we say for brevity that $G$ is bounded-by-abelian-by-bounded if it has a structure as in Neumann's theorem.  The reader is referred to the papers \cite{AS,AS2,bgmn,DS-commprob,DLMS-prob,DMS-Syl_prof,DLMS-finite,DMS-approx,eberhard,gr} 
 for further results on commuting probability. The recent article \cite{DLMS-finite} handles finite groups $G$ in which for every pair of distinct primes $p,q\in\pi(G)$ there is a Sylow $p$-subgroup $P$ and a Sylow $q$-subgroup $Q$ such that $\pr(P,Q)\geq\ep$. Among other results, it was shown that the index $|G:F_2(G)|$ is bounded in terms of $\ep$ only.  Here $F_2(G)$ stands for the second term of the upper Fitting series of $G$.

 In this paper we consider finite groups $G$ such that for some prime $p\in\pi(G)$ we have $\pr(P_1,P_2)\geq\ep>0$ whenever $P_1,P_2$ are Sylow $p$-subgroups of $G$. As usual, $O_p(G)$ stands for the maximal normal $p$-subgroup of $G$.

\begin{theorem}\label{main1} Let $G$ be a finite group and $p\in\pi(G)$. Assume that $\pr(P_1,P_2)\geq\ep>0$ for every two Sylow $p$-subgroups $P_1,P_2$ of $G$. Then the order of a Sylow $p$-subgroup of $G/O_p(G)$ is bounded in terms of $\ep$ only. Moreover $O_p(G)$ is bounded-by-abelian-by-bounded.
\end{theorem}

A curious corollary of the above result is that if the condition holds for every prime in $\pi(G)$, then the structure of $G$ is roughly similar to that of a group $K$ with $\pr(K)\geq\ep$.

\begin{corollary}\label{cor1} Let $G$ be a finite group such that for every $p\in\pi(G)$ and every two Sylow $p$-subgroups $P_1,P_2$ of $G$ we have $\pr(P_1,P_2)\geq\ep>0$. Then $G$ is bounded-by-abelian-by-bounded. 
 \end{corollary}

In the course of the work on the above results, the following challenging question was raised. \medskip

{\it Let $P$ be a $p$-subgroup of a finite group $G$ and assume that $\pr(P,P^x)\geq\ep>0$ for every $x\in G$. Is the order of $P$ modulo $O_p(G)$ bounded in terms of $\ep$ only?}\medskip

The answer to this turns out to be negative -- see  Example \ref{example}. Yet, we have established several positive results with respect to the above question.

\begin{theorem}\label{main2} Let $G$ be a finite group containing a $p$-subgroup $P$ such that $\pr(P,P^x)\geq\ep>0$ for every $x\in G$. Suppose that the composition factors of $G$, which are isomorphic to simple groups of Lie type in characteristic $p$, have Lie rank at most $n$. Then the order of $P$ modulo $O_p(G)$ is bounded in terms of $n$ and $\ep$ only.
\end{theorem}

This shows that in the case of $p$-soluble groups the answer to the above question is positive.
\begin{corollary}\label{cor2} Assume that a finite $p$-soluble group $G$ contains a $p$-subgroup $P$ such that $\pr(P,P^x)\geq\ep>0$ for every $x\in G$. Then the order of $P$ modulo $O_p(G)$ is bounded in terms of $\ep$ only. \end{corollary}

Let $E_p(H)$ denote the product of the components of order divisible by $p$ in a finite group $H$. Write $O_{p,e_p}(H)$ for the full inverse image of $E_p(H/O_p(H))$ and $O_{p,e_p,p}(H)$ for the full inverse image of $O_p(H/O_{p,e_p}(H))$.

\begin{theorem}\label{main3} Assume that a finite group $G$ contains a $p$-subgroup $P$ such that $\pr(P,P^x)\geq\ep>0$ for every $x\in G$. Then the order of $P$ modulo $O_{p,e_p,p}(G)$ is bounded in terms of $\ep$ only. \end{theorem}

The above theorems can be easily extended to the case where the order of the subgroup is not necessarily a prime power. This is because, given a subgroup $K\leq G$ such that $\pr(K,K^x)\geq\ep>0$ for every $x\in G$, we can apply the above results to the Sylow subgroups of $K$. For example, the following holds.

\begin{corollary}\label{cor3}  Let a finite group $G$ contain a subgroup $K$ such that $\pr(K,K^x)\geq\ep>0$ for every $x\in G$.
\begin{enumerate} 
\item If $K$ is a Hall subgroup of $G$, then the order of $K$ modulo the Fitting subgroup $F(G)$ is bounded in terms of $\ep$ only.
\item If $n$ is the maximum of Lie ranks of composition factors of $G$, then the order of $K$ modulo 
  $F(G)$ is bounded in terms of $n$ and $\ep$ only. 
\end{enumerate}
\end{corollary}

The concept of commuting probability can be naturally extended to compact groups, using the normalized Haar measure (see Section 6 for details). 

In particular, Theorem \ref{main1} can be naturally extended to profinite groups $G$ such that there is a positive number $\epsilon$ and a prime $p\in\pi(G)$ with the property that $\pr(P_1,P_2)\geq\ep>0$ for every two Sylow $p$-subgroups $P_1,P_2$ of $G$ (see Proposition \ref{profinite-easy}). Here $\pi(G)$ denotes the set of primes dividing the order of the profinite group $G$, which is a supernatural number.

On the other hand, in the realm of profinite groups, it is more natural to handle the situation where the probabilities are positive, but not necessarily bounded away from zero. 
 
In this paper we prove the following profinite variation of Theorem \ref{main1}.

\begin{theorem}\label{main4}
Let $G$ be a profinite group and $p\in\pi(G)$. Assume that $\pr(P_1,P_2)>0$ for every two Sylow $p$-subgroups $P_1,P_2$ of $G$. Then $O_{p,p'}(G)$ is open. Moreover $O_p(G)$ is virtually abelian.
\end{theorem}

As usual, we say that a profinite group has a property virtually if it has an open subgroup with that property. 

In the next section we collect some auxiliary results needed for  the proofs of the main theorems. In Section \ref{useful} we prove several propositions that provide important technical tools employed throughout the paper. The crucial Section \ref{simple} handles almost simple groups. Section \ref{thms1.1-1.3} provides proofs of Theorems \ref{main1} and \ref{main2} as well as corollaries.  A proof of Theorem \ref{main3} is given in Section \ref{thm3}.  Finally, Section \ref{thm4} deals with profinite groups.

\section{Preliminaries} 

We start with a lemma which lists some well-known properties of coprime actions (see for example \cite[Ch.~5 and 6]{go}).  In the sequel the lemma will often be used without explicit references.

\begin{lemma}\label{cc}
Let  $A$ be a group of coprime automorphisms of a finite group $G$. Then
\begin{enumerate}
\item[(i)] $G=[G,A]C_{G}(A)$. If $G$ is abelian, then $G=[G,A] \times C_{G}(A)$.
\item[(ii)] $[G,A,A]=[G,A]$. 
\item[(iii)] $C_{G/N}(A)=NC_G(A)/N$ for every $A$-invariant normal subgroup $N$ of $G$.
\item[(iv)] If $[G/\Phi(G),A]=1$, then $[G,A]=1$.
\end{enumerate}
\end{lemma}

The next lemma will be useful.
 
\begin{lemma}\label{covers}
Let $G$ be a group having subgroups $N_1, \dots ,N_s$ 
such that  $\cap_{i=1}^s N_i = 1$ but $\cap_{j \in J} N_j \neq 1$ for every proper subset $J$ of $\{1, \dots , s\}$.
Then $\cup_{i=1}^s N_i \neq G$.
\end{lemma}
\begin{proof}
Suppose $G$ is a counter-example with $s$ as small as possible. Note that $N_s$ and its subgroups $U_i =N_i \cap N_s$, for $i \in \{1, \dots , s-1\}$, satisfy the hypothesis. By minimality, $N_s \neq  \cup_{i=1}^{s-1} U_i$. Moreover $X= \cap_{1 \le i \le s-1} N_i$ is nontrivial. Let us fix $1 \neq x \in X$ and $y \in N_s \setminus \cup_{i=1}^{s-1} U_i$. If $xy \in N_s$, then $x \in X \cap N_s =1$, a contradiction. If $xy \in N_i$, for some $i <s$,  then $y \in N_i \cap N_s =U_i$, a contradiction again. Therefore $xy \notin  \cup_{i=1}^s N_i$. 
\end{proof}

We will record here Neumann's theorem \cite{neumann} that was mentioned in the introduction. It describes the structure of finite groups with high commuting probability.

\begin{theorem}\label{neumann}
Let $G$ be a finite group with $\pr(G)\geq\ep>0$. Then $G$ has a normal subgroup $H$ such that $|H'|$ and $|G:H|$ are both $\ep$-bounded.
\end{theorem}

The following lemmas on commuting probability for subgroups are taken from \cite{DLMS-finite}. 

\begin{lemma}\label{lem:2.2}%\cite[Lemma 2.2]{DLMS-finite} 
 Let $G$ be a finite group and let $H, K$ be subgroups of $G$.  Then 
\begin{enumerate}
\item If $N$ is a normal subgroup of $G$, then $\pr (HN/N,KN/N) \ge \pr (H,K)$.
\item If $H_0 \le H$, then $\pr (H_0,K) \ge \pr (H,K)$. 
\item If $G=G_1 \times G_2$, $H_i \le G_i$   and $K_i \le G_i$, then 
\[\pr (H_1 \times H_2, K_1 \times K_2) = \pr (H_1, K_1) \pr(H_2, K_2).\] 
\item\label{3/4} If $\pr (H,K) >3/4$, then $[H,K]=1$.
\end{enumerate}
\end{lemma}

Note that Item \ref{3/4} is straightforward from Lemma 2.5 of \cite{DLMS-finite}, which states that, whenever $[H,K] \neq 1$, 
  $\pr (H,K) \le (n+m-1)/(nm)$ where $n=|H:C_H(K)|$ and $m=\min_{x \in H\setminus C_H(K)} |K:C_K(x)|$ (see also \cite[Remark 2.6]{DLMS-finite}).

\begin{lemma}\label{lem:2.9}
 Let $G$ be a finite group and let $H, K$ be subgroups of $G$ with $\pr (H,K) \ge \epsilon >0 $. Then 
 there exists a subset $X$ of $H$ such that for  $H_0=\langle X\rangle$ the following holds. 
 \begin{enumerate}
\item $|H:H_0| \le 2/\epsilon-1$; 
\item $|K: C_K(x)| \le 2/\epsilon$ for every $x \in X$;
\item 
  $|K: C_K(x)| \le (2/\epsilon)^{6/\epsilon}$  for every $x \in H_0$.
\end{enumerate}
\end{lemma}

A large part of this paper addresses the following hypothesis.

\begin{hypothesis}\label{basic} Let $\ep$ be a positive number and $P$ a $p$-subgroup of a finite group $G$ such that $\pr(P,P^x)\geq\ep$ for every $x\in G$.
\end{hypothesis}

We remark that the structure of the subgroup $P$ is as described in Neumann's Theorem \ref{neumann}. Moreover,  Lemma \ref{lem:2.2} (2) shows that if $H\leq P$, then 
$\pr(H,H^x)\geq\ep$ for every $x\in G$.

The following lemma shows that Hypothesis \ref{basic} is preserved under taking subgroups and quotients. Throughout the paper it is often used without explicit mention. 
\begin{lemma}\label{trivial}
Assume Hypothesis \ref{basic}.
\begin{enumerate} 
\item If $H \le G$, then the pair $(H, P \cap H)$ satisfies Hypothesis \ref{basic}. 
\item If $N \unlhd G$, then the pair $(G/N, PN/N)$ satisfies Hypothesis \ref{basic}.
\end{enumerate}
\end{lemma}
\begin{proof}
It is a straightforward application of Lemma \ref{lem:2.2}. 
\end{proof}

In what follows we often say that a quantity is bounded to mean that it is bounded in terms of $\ep$ only. More generally, we say that a quantity is $(a,b,c\dots)$-bounded to mean that it is bounded in terms of the parameters $a,b,c\dots$.

\begin{lemma}\label{lemH_0} 
Assume Hypothesis \ref{basic}. There is a bounded number $t$ such that for every $g\in G$ there is a normal subgroup $P_g$ of $P$ with the properties that $|P : P_g|\le t$ and 
$|P^g: C_{P^g} (x)| \le t$ whenever $x \in P_g$.

Conversely, let $t$ be a positive integer and $P$ a subgroup of a finite group $G$ such that for every $g\in G$ there is a normal subgroup $P_g$ of $P$ with the properties that $|P : P_g|\le t$ and 
$|P^g: C_{P^g} (x)| \le t$ whenever $x \in P_g$. Then $\pr(P,P^g)\geq 1/t^2$ for every $g\in G$.
\end{lemma}
\begin{proof} Lemma \ref{lem:2.9} 
 guarantees that there exists a subgroup $Q_g$ of $P$ such that $|P:Q_g| \le 2/\epsilon-1$ and $|P^g : C_{P^g}(x)| \le (2/\epsilon)^{6/\epsilon}$  for every  $x \in  Q_g$. 
 Then the subgroup $P_{g} = \cap_{y \in P} Q_{g}^y$ is a normal subgroup of $P$ with the required properties. 
 
The converse statement is straightforward. 
 \end{proof}

\begin{lemma} \label{lift} For a prime $p$ and a positive number $\epsilon$ there is a positive number $\delta=\delta(p,\epsilon)$ with the following property. Let $G$ be a finite group and $Z$ a cyclic subgroup of $Z(G)$. Set $\bar{G}=G/Z$. Assume that $G$ contains an elementary abelian $p$-subgroup $P$ such that $\pr(\bar{P},\bar{P}^{\bar{g}})\geq\epsilon$ for every $\bar{g}\in\bar{G}$. Then $\pr({P},P^{g})\geq\delta$ for every $g\in G$. 
\end{lemma}
\begin{proof}
 Observe that $P\cap Z$ has order at most $p$. In view of Lemma \ref{lemH_0} there is an $\ep$-bounded integer $t$ such that  the subgroup $\bar{P}^g$ contains a subgroup $\bar{P_{0}}$ of index at most $t$ such that
 $|\bar P^g:C_{\bar P^g}(\bar x)| \le t$, for every $\bar x \in \bar{P_{0}}$.  Let $P_{0}$ be the preimage of $\bar{P_{0}}$ in $P$.
 
Note that $P_{0}$ has index at most $t$ in $P$. 
 Let $x\in P_{0}$. It follows that $\bar{x}$ has centralizer of index at most $t$ in  $\bar P^g$. Therefore $P^g$ contains a subgroup $K$ of index at most $t$ such that $[K,x]\leq Z$.  
 So $\langle K,x\rangle$ is nilpotent of class at most 2. As $\langle K,x\rangle$ is generated by elements of order $p$, we deduce that the commutator subgroup $\langle K,x\rangle'$ has exponent dividing $p$. Since $\langle K,x\rangle'\leq Z$ and since $Z$ is cyclic, it follows that the order of $\langle K,x\rangle'$ is at most $p$. Hence $C_K(x)$ has index at most $p$ in $K$ and index at most $pt$ in $P^g$. So every element of $P_{0}$ has centralizer of index at most $pt$ in $P^g$. Now the `converse' part of Lemma \ref{lemH_0} shows that there is $\delta$ with the required properties.
 \end{proof}
 
The following lemma provides a crucial technical tool which is used throughout the paper on multiple occasions.

\begin{lemma}\label{lemQ} 
Assume Hypothesis \ref{basic} and let $a \in G$. Then $P$ contains a normal subgroup $Q_a$ of bounded index with the property that, for every $g\in Q_a$, there exists a subgroup $D_{a,g}$ of bounded index in $Q_a$ such that $\langle g^P\rangle$ centralizes $[a,D_{a,g}]$.
\end{lemma}
\begin{proof}
By Lemma \ref{lemH_0} applied to the subgroups $P$ and $P^a$, respectively, there exist two normal subgroups $Q_1$ and $Q_2$ of index at most $t$ in $P$  such that the elements in $Q_1$ have centralizers of index at most $t$ in $P$ and the elements in $Q_2$ have centralizers of index at most $t$ in $P^a$. Let $Q=Q_1\cap Q_2$ and for an element $g\in Q$ set $L=\langle g^{P} \rangle$. Note that $L $ is contained in $Q$. Moreover, as $g \in Q$, the subgroup $L$ is generated by boundedly many conjugates of $g$, all of which have centralizers of index at most $t$ in both $P$ and $P^a$. Thus $C_1=C_{P} (L)$ and $C_2=C_{P^a} (L)$ have bounded index in $P$ and $P^a,$ respectively. Let $D=C_2^{a^{-1}} \cap C_1 \cap \Q$ and note that $D$ is normal in $P$. The subgroup $D$ has bounded index in $Q$ and moreover  $L$ centralizes both $D$ and $D^a$. It follows that  $L$ centralizes the subgroup  $[a,D]$. 
\end{proof}

\begin{lemma}\label{exp-bounded}
Assume Hypothesis \ref{basic}. Then there is a bounded integer $e$ such that $P^e \le F(G)$.\end{lemma}
\begin{proof} 
Choose any $x\in G$ and let $t$ and $P_x$ be as in Lemma \ref{lemH_0}.  There is a bounded integer $m$ such that $P^m\le P_x$. Therefore $P^m\le\cap_{x\in G}P_x$, which implies that for every element $g$ in $P$ the power $g^m$ has a centralizer of  index at most $t$ in every conjugate $P^x$. It follows that there exists a bounded integer $i$ such that $((g^m)^x)^i$ centralizes $g^m$. We conclude that $g^{im}$ commutes with all of its conjugates and therefore $g^{im} \in F(G)$. 
\end{proof}

\begin{lemma}\label{p-bounded}
Assume Hypothesis \ref{basic}. If $p$ divides the order of $G/O_p(G)$, then  $p$ is bounded. 
\end{lemma}
\begin{proof}  Let $e$ be as in Lemma \ref{exp-bounded}. We know that $P^e\le F(G)$. Since $p$ divides the order of $G/O_p(G)$, it follows that $p\leq e$.
\end{proof}

\section{Useful propositions}\label{useful}

Throughout this section we assume Hypothesis \ref{basic}. Recall that we say that a quantity is bounded if and only if it is bounded in terms of $\ep$ only.

\begin{proposition}\label{soluble} Suppose that $G=MP$, where $M=F(G)$ and assume that $O_p(G)=1$. Then the order of $P$ is bounded.
\end{proposition}

\begin{proof}
Note that $F(G)$ is a $p'$-group. We can pass to the quotient $F(G)/\Phi(F(G))$ over the Frattini subgroup and without loss of generality assume that $M$ is a direct sum of elementary abelian subgroups. By Maschke's theorem
\[ M = M_1 \oplus \dots \oplus M_s, \]
 where each $M_i$ is an irreducible $P$-module. Since $O_p(G)=1$, we have $C_P(M)=1$. Moreover $P$  acts faithfully on $[M,P]$ so we may assume that $M=[M,P]$ and therefore $C_M(P)=1$.
 
Choose elements $1\neq a_i\in M_i$ for $i=1,\dots,s$ and set $a=a_1+\ldots+a_s$. 

Let  $Q=Q_a$ be the normal subgroup of $P$ as in Lemma \ref{lemQ}. As $Q$ has bounded index in $P$, it is sufficient to show that the order of $Q$ is bounded. We can thus assume that $Q \neq 1$.  

Let $K_i=C_Q(M_i)$. Note that $\cap_{i=1}^s K_i=1 $. Let $J$ be a minimal subset of the set $\{1,\dots,s\}$ such that $\cap_{j \in J} K_j=1$. Observe that $Q$  acts faithfully on $\sum_{j\in J}M_j$. In view of Lemma  \ref{covers}, $\Q \neq \cup_{j \in J} K_j$. Therefore there is $g \in Q$ such that $g$ acts nontrivially on each $M_j$ for $j\in J$. 

Let $D=D_{a,g}$ be as in Lemma \ref{lemQ} and $L=\langle g^P\rangle$. According to Lemma \ref{lemQ}, the subgroup $L$ centralizes $[a,D] $ and, in particular, $L$ centralizes each of the subgroups $[a_i,D]\leq M_i$. 
 
Let $j \in J$. Since $L$ is normal in $P$, the centralizer $C_{M_j}(L)$ is a $P$-submodule of $M_j$. As each $M_j$ is irreducible, we have that either $L$ centralizes $M_j$ or $C_{M_j}(L)=1$.  As  $g$ acts nontrivially on  $M_j$, we deduce that  $C_{M_j}(L)=1$.  
It follows that $[a_j,D]=1$. So  $1 \neq a_j \in C_{M_j}(D)$ and $ C_{M_j}(D)$ is nontrivial.  Because $D$ is normal in $P$ and $M_j$ is irreducible, we conclude that  $D$ centralizes $M_j$. 

 Therefore $D \le \cap_{j \in J} K_j=1$. Since the index $|Q:D|$ is bounded, so is the order of  $\Q$. This completes the proof. 
\end{proof}

We say that a group is semisimple if it is a direct product of nonabelian simple groups.

\begin{proposition}\label{orbit} Suppose that $G=MP$, where $M$ is a normal semisimple subgroup of $G$. Then the size of every $P$-orbit under the natural permutational action of $P$ on the simple factors of $M$ is bounded.
\end{proposition} 
\begin{proof} Without loss of generality, we can assume that   $M= S_1 \times \cdots \times S_l$, where each $S_i$ is a finite nonabelian simple group and $\{S_1, \ldots, S_l\}$ is a $P$-orbit. Let $P^*$ be the stabilizer of $S_1$. We need to show that the index of $P^*$ in $P$ is bounded. For an element $x\in S_1$  set $P_1=P$ and $P_2=P^x$. Let  $Q=Q_x$ be the normal subgroup of $P$ as in Lemma \ref{lemQ}. Suppose $Q$ is not contained in $P^*$ and choose $g\in Q$ such that $g\not\in P^*$. Then, by Lemma \ref{lemQ}, $P$ contains a subgroup $D$ of bounded index such that $g$ centralizes $[x,D]$. If $D\leq P^*$, we are done. Otherwise, pick $d\in D$ such that $d\not\in P^*$. For simplicity assume that $S_2=S_1^d$. We see that $[x,d]$ is a nontrivial element of $S_1S_2$. So $g$ has nontrivial centralizer in $S_1S_2$ and therefore $\{S_1,S_2\}$ is a $g$-orbit. This argument can be applied for every element of $D$. Since $S_1$ belongs to only one $g$-orbit, it follows that $S_1^y\in \{S_1,S_2\}$ for every element $y\in D$, that is $\{S_1,S_2\}$ is a $D$-orbit. Hence, the index of $D\cap P^*$ in $D$ is at most $2$ and therefore the index of $P^*$ in $P$ is bounded.
\end{proof}

\begin{proposition}\label{trickB} Suppose that $G=MP$, where $M$ is a normal semisimple subgroup of $G$, and assume that $O_p(G)=1$. 
 Let $C$ be a constant such that 
 $|P:C_P(S)|\leq C$ for every simple factor $S$ of $M$. Then  the order of $P$ is $(C,\epsilon)$-bounded.
\end{proposition}
\begin{proof} 
Note that $C_P(M)=1$.
Write $M= M_1\times \cdots\times M_l$, where each $M_i$ is a minimal normal subgroup of $G$. Let $J$ be a minimal subset of the set $\{1,\dots,l\}$ such that $P$  acts faithfully on $\prod_{j\in J}M_j$. Replacing $M$ by $\prod_{j\in J}M_j$ we simply assume that $J=\{1,\dots,l\}$.

It follows from  Proposition \ref{orbit} that every $M_i$ is a product of boundedly many simple factors so it is sufficient  to bound $l$.

For every $i$ set $N_i=\prod_{j\in J\setminus\{i\}}M_j$ and $Q_i=C_P(N_i)$. Choose an element $y_i$ of prime order in $Q_i$ and set $\D=\langle y_1,\dots, y_l\rangle$. It is enough to show that $\D$ has bounded order. 

For each $i$, the element $y_i$ induces on $M_i$  an automorphism $\phi_i$ of order $p$. Observe that the group of automorphisms induced by $\D$ on $M_i$ is $\langle\phi_i\rangle$. Pick an element $a_i \in M_i$ such that $[a_i, \phi_i,\phi_i] \neq 1$. Such an element $a_i$ exists for every $i$, because 
 otherwise $y_i$ would be a left Engel element of $G$ lying in the Fitting subgroup of $G$ (see  \cite[12.3.7]{robinson}),  
which immediately leads to a contradiction. 

Let $a= \prod_i a_i$ and let  $Q=Q_a$ be the normal subgroup of $\D$ as in Lemma \ref{lemQ}. As $Q$ has bounded index in $\D$,  it is sufficient to bound the order of $Q$.
 
Let $I_0$ be a minimal subset of $\{1,\dots,l\}$ such that $\cap_{i\in I_0} C_Q(M_i)=1$. It follows from Lemma \ref{covers} that there exists an element $g\in \Q$ such that $g \notin C_{\Q}(M_i)$ for every $i\in I_0$. 
  
  By Lemma \ref{lemQ}, there exists a subgroup $D$ of bounded index in $Q$ such that $g$ centralizes  the subgroup $[a,D]$. 
  So $g$ centralizes  also the projections $[a_i ,D]$ of  $[a,D]$ on $M_i$. 
  The group of automorphisms induced by $Q$ on $M_i$ is $\langle \phi_i \rangle$, which has prime order. 
  So if $[a_i,D] \neq 1,$ then $[a_i,D]$ contains the element $[a_i,\phi_i]$. Moreover, $[a_i,\phi_i, g] \in [a_i,D, g]=1$. 
  
On the other hand, the automorphism induced by $g$ on $M_i$ is a nontrivial power of $\phi_i$. Therefore the fact that $[a_i, \phi_i, g]=1$ implies that $[a_i,\phi_i, \phi_i]=1$, which contradicts the choice of $a_i$. 
  
This means that $[a_i,D] =1$ for every $i$. If the automorphism $\phi_i^{r_i}$ induced by $D$ on $M_i$ is nontrivial, then $\phi_i^{r_i}$ centralizes $a_i$ and so does $\phi_i$, a contradiction again. Therefore $D$ centralizes all $M_i$ for $i\in I_0$ and $Q$ has bounded order. The lemma follows.
\end{proof}

\section{The case where $G$ is almost simple}\label{simple} 

We start this section by giving an example showing that under Hypotheses \ref{basic} the order of $P$ can be arbitrarily large even if $O_p(G)=1$.

\begin{example}\label{example}{\rm Let $G= {\mathrm{SL}}(n,p) = \mathrm{SL}(V)$ for  $n\geq3$, and let $A$ be a hyperplane in $V$. 

Let $P$ be the elementary abelian subgroup of order $p^{n-1}$ which is trivial on both $A $ and $V/A$, and choose $x\in G\setminus N_G(P)$.
 
Note that $P$ and $P^x$ intersect trivially and the subgroup $\langle P,P^x\rangle$  has an abelian subgroup $Q$ of order $p^{2n-2}$ such that $Q\cap P$ and $Q\cap P^x$  each have index $p$ in $P$ and $P^x$ respectively.  We see that $\pr(P,P^x)\geq1/p^2$. Since $n$ can be arbitrarily large, it follows that there is no bound on the order of $P$.} \qed
\end{example}

The main goal of this section is to prove the following theorem.

\begin{theorem}\label{almostsimple} Assume Hypothesis \ref{basic} with $G$ almost simple. Then the order of $P$ is bounded unless the socle $S$ is of Lie type in characteristic $p$. If $S$  is of Lie type in characteristic $p$ of rank $n$, then the order of $P$ is $(n,\ep)$-bounded.
\end{theorem}

Given a finite group $G$ and a subgroup $H$, let $$e_G(H)=|G|^{-1} \sum_{g \in G}\pr(H,H^g)$$ be the average value of $\pr(H,H^g),\  g\in G$. Observe that the probability that an element $x$ commutes with a random conjugate of an element $y$ is $\frac{|y^G\cap C_G(x)|}{|y^G|}$. We write $H^\#$ to denote the set of nontrivial elements of $H$.

\begin{lemma}\label{f} 
Let $f \le 1$ and suppose that $H$ is a subgroup of a finite group $G$ such that for every pair of nontrivial elements $x,y\in H$ the probability that $x$ and a random conjugate of $y$ commute is at most $f$. Then   \[  e_G(H)  < 2/|H| + f. \]
\end{lemma}  
\begin{proof}  Let $n=|H|$, thus 
\begin{eqnarray*}
n^2 |G| e_G(H) &=& n^2  \sum_{g \in G} \pr(H,H^g)=\sum_{g \in G}  \sum_{x \in H} |C_{H^g}(x)|\\
&=& \sum_{g \in G}  \left( n + (n-1) +  \sum_{x \in H^\#} |C_{H^{g\#}}(x) |\right)\\
&=&(2n-1) |G| +\sum_{x \in H^\#} \left(\sum_{g \in G}  |C_{H^{g\#}}(x) |\right).
\end{eqnarray*}

Now 
\begin{eqnarray*}
 \sum_{g \in G}  |C_{H^{g\#}}(x) | &=& \sum_{g \in G} |\{ y \in H^\# \mid [x, y^g] =1\} |\\
&=& |\{ (g, y) \in G \times H^\# \mid [x, y^g]=1\}|\\
&=&\sum_{y \in H^\#} | \{ g \in G \mid [x, y^g] =1\}|.
\end{eqnarray*}

Note that $|\{w\in G \mid y^w=y^g\}|=|C_G(y)|$.
Now fix $1 \ne x, y  \in H$. We have
\begin{eqnarray*}
|\{(x,g) \mid g\in G, [x,y^g]=1\}| &=&|\{(x,y^g) \mid   [x,y^g]=1 \}|\,|C_G(y)| \\ 
&=& |y^G\cap C_G(x)|\,|C_G(y)|\\
&=&\frac{|y^G\cap C_G(x)|}{|y^G|}\,|y^G||C_G(y)| \le f |G|.
\end{eqnarray*}

Thus, the total number of commuting pairs  $(x,z^g)$ as $z$ ranges over $H^\#$ and $g$ ranges over $G$
is at most  $(n-1)f|G|,$  as there are $n-1$ choices for $z$: 
\[  | \{ (g, y) \in G \times H^\# \mid [x, y^g]=1\}| \le (n-1)f|G|. \]

Thus, 
\[ \sum_{g \in G}  \left( \sum_{x \in H^\#} |C_{H^{g\#}}(x) | \right)\le  (n-1)^2 f|G|. \]

Putting all together we have 
 \[ n^2 |G| e_G(H) \le  (2n-1) |G| + (n-1)^2 f|G| , \]
which gives 
 \[  e_G(H) \le \frac{2n-1}{n^2} + \frac{(n-1)^2 f}{n^2} < 2/n + f, \]
as claimed. 
  \end{proof}  

\begin{lemma} \label{regular} Let $n=p^a \ge 5$ and let $G=S_n$ be the symmetric group of degree $n$. Let $P$ be a regular elementary abelian $p$-subgroup of $G$. Then $e_G(P) \rightarrow 0$ as $a \rightarrow \infty$.
\end{lemma}

\begin{proof} If $1 \ne x \in P$, then $x$ is a fixed-point-free permutation of order $p$. Observe that $P^\#\subset x^G$.   
Let $f$ be the probability that $x$ and a random conjugate  of $x$ commute. Thus, $f=|x^G \cap C_G(x)|/|x^G|$.  
Note that $C_G(x)$ has 
 a normal elementary abelian subgroup of order $p^{p^{a-1}}$ with the quotient isomorphic to $S_{n/p}$.   

If $p$ is odd, observe that
\[
\lim_{a \rightarrow \infty}  |C_G(x)|/|x^G| =\lim_{a \rightarrow \infty}  |C_G(x)|^2 /|G| = 0.
\]
Indeed, we have $|C_G(x)|^2 /|G|=|S_{n/p}|^2p^{2p^{a-1}}/(p^a)!$ and
this tends to zero when $a$ goes to infinity.

If $p=2$, we note that $|x^G \cap C_G(x)| < 2^{n/2}|I_2(S_{n/2})|$ where $I_2(H)$ is the number
of involutions in $H$.   As  the number of involutions in $S_m$ is asymptotic to 
\[\frac 1{\sqrt 2}m^{\frac m2}e^{-\frac m2+\sqrt m-\frac 14}\]
(cf. \cite[p. 155--158]{wilf}), 
 it follows that 
\[f=\frac{|x^G \cap C_G(x)|}{|x^G|}\le \frac{2^{n/2}|I_2(S_{n/2})|\,|C_G(x)|}{|G|}\le \frac{2^{n/2}|I_2(S_{n/2})|2^{n/2}|S_{n/2}|}{n!},\]
which tends to zero.

Lemma \ref{f} tells us that $e_G(P)  < 2/n + f$ and so the result follows. 
 \end{proof} 

\begin{remark}
{\rm By virtue of Lemma \ref{exp-bounded}, under the hypotheses of Theorem \ref{almostsimple} there is a bounded integer $e$ such that $P^e=1$. In particular $p\leq e$. Since the order of a finite group is bounded in terms of the maximum of its abelian subgroups 
 (see for instance \cite[Theorem 5.2]{Aiv}),
  we may replace $P$ by an abelian subgroup and simply assume that $P$ is abelian. Since $P$ has bounded exponent, we may again replace $P$ by an elementary abelian subgroup and assume that $P$ is elementary abelian. 
}
\end{remark}

\begin{remark}\label{rem3/4}
{\rm 
Note that whenever $P$ is a $p$-subgroup of a finite group $G$ such that $P\not\leq O_p(G)$ there is $g \in G$ such that $P$ does not commute with $P^g$. Hence $\pr(P,P^g) \le 3/4$ by Lemma \ref{lem:2.2}(\ref{3/4}). 
}
\end{remark}
 
\begin{lemma}\label{alter}  Theorem \ref{almostsimple} holds for alternating and symmetric groups.
\end{lemma}

\begin{proof}  Assume Hypothesis \ref{basic} with $G=S_n$, the symmetric group of degree $n,$ and $P$ an elementary abelian $p$-subgroup of $G$. Let $m$ the size of the largest orbit for $P$. 

 We wish to show that $m$ is bounded. So assume that $m\geq5$. Let $\Omega$ be an orbit of size $m$, and let $Q$ be the kernel of $P$ on $\Omega$. Since $P$ is elementary abelian, there exists $P_0$ such that $P = P_0 \times Q$.
   Observe that $P_0$ acts regularly on $\Omega$. By Lemma \ref{regular}, there exists $n_0$ such that if $m > n_0$, then $e_{S_m}(P_0) < \ep$. 
  In particular for at least one element $g\in S_m$, we have $\pr(P_0, P_0^g) <\epsilon$. Thus we can find $\tilde g \in S_n$, such that $\tilde g$ acts as $g$ on $\Omega$, and so 
  \[ \pr(P,P^{\tilde g}) \le \pr(P_0, P_0^g) < \epsilon, \] 
  against our assumptions. Therefore, $m\le n_0$. 
  
Let $s$ be the minimum number such that there are $s$ $P$-invariant disjoint subsets $\Omega_1,\dots,\Omega_s$, each of size at least $5$, such that $P$ acts faithfully on their union.

Now we need to bound $s$. We claim that there exists $g\in G$ such that $\pr(P,P^g) < (3/4)^s$. This will be proved by induction on $s$. If $s=1$, the result is immediate from Remark \ref{rem3/4}; 
 so assume that $s\geq2$.

Let $Q_s$ be the kernel of the action of $P$ on $\Omega_s$ and write $P=P_s\times Q_s$, with $P_s$ acting faithfully on $\Omega_s$. 
By Remark \ref{rem3/4} 
 there is $g_s\in G$ with support in $\Omega_s$ such that $\pr(P_s,P_s^{g_s}) <3/4$. Observe that $Q_s$ acts faithfully on the union of $\Omega_1,\dots,\Omega_{s-1}$. By induction there exists $g_1 \in S_n$, with support disjoint from that of $g_s$, such that $\pr(Q,Q^{g_1}) < (3/4)^{s-1}$. Therefore, for $g=g_1g_2$, we have $\pr(P,P^g) < (3/4)^s$, as claimed.   
Thus, $(3/4)^s\geq\ep$ and so $s$ is bounded. 

Note that $P$ can be embedded in a product of at most $s$ symmetric groups of degree at most $m$. Hence $|P|$ is bounded and the case where $G=S_n$ is done. 

To deal with the case where $G=A_n$ is the alternating group, observe that $P$ contains a subgroup $R$ of index at most $p^2$ such that any conjugate of $R$ in $S_n$ is in fact a conjugate in $A_n$. So the case of $A_n$ is immediate from the above.
\end{proof} 

\begin{lemma}\label{lie}  Let $G$ be an almost simple group of Lie type defined over a field of size $q$ and Lie rank at least $2$. If $q$ is sufficiently large, then for every subgroup $H$ of $G$ we have $e_G(H) <  2/|H| + 4/3q$.
In any case, $e_G(H)<2/|H|+f(q)$, where $f(q)$ tends to $0$ as $q$ goes to infinity.
\end{lemma}  

\begin{proof}   Let $x, y$ be nontrivial elements of $G$. 
 The group $G$ naturally acts on the conjugacy class $y^G$ and the kernel of this action is trivial. Let $\Sigma=\{\B_1,\dots,B_k\}$ be a system of maximal blocks of imprimitivity, each of order $m$, so that $|y^G|=m|\Sigma|$ and $G$ acts faithfully and primitively on $\Sigma$. 

If $x\in G$ fixes $B_i$, then $x$ centralizes at most $m$ elements in $B_i$. Therefore  $|y^G\cap C_G(x)|\le m|fix_{\Sigma}(x)|.$  
By  \cite[Theorem 1]{LS} 
 $ |fix_{\Sigma}(x)|/|\Sigma|\le 4/3q$ 
 apart from some known exceptions that have rank $1$ or small order. In any case,  
 $ |fix_{\Sigma}(x)|/|\Sigma| \le f(q)$, 
  where $f(q)$ tends to $0$ as $q$ goes to infinity.
It follows that 
\[\frac{ |y^G\cap C_G(x)|}{|y^G|}\le   \frac{ m |fix_{\Sigma}(x)|}{m|\Sigma|}\le f(q),\]
 so the probability that  $x$ and a random conjugate of $y$ commute is bounded by $f(q)$. 

Now using Lemma \ref{f} we deduce that 
   \[  e_G(H)  < 2/|H| + f(q), \] as required.
\end{proof}

\begin{lemma}\label{qbounded}  There is a bounded positive integer $q_0$ such that, if under Hypothesis \ref{basic} the group $G$ is an almost simple group of Lie type defined over a field of size $q>q_0$, then $|P| \le 4/\ep$.
\end{lemma}
\begin{proof} Let $f(q)$ be a function as in Lemma \ref{lie}. There exists a number $q_0$ depending only on $\ep$ such that if $q>q_0$, then $f(q) < \ep/2$. It follows from Lemma \ref{lie} that $e_G(P) < 2/|P|+\ep/2$. There is an element $g\in G$ such that $\pr(P, P^g)\le e_G(P)$, while $\pr(P, P^g) \ge \ep$  by hypotheses. 
Putting all this together we get 
\[ \ep \le \pr(P, P^g) \le e_G(P) <  2/|P| + \ep/2,\] 
whence $2/|P| \ge \ep/2$ and so $|P| \le 4/\ep$. 
\end{proof}

\begin{lemma}\label{difchara} Assume Hypothesis \ref{basic} with $G$ a simple group of Lie type in characteristic other than $p$. Then the order of $P$ is bounded.
\end{lemma}
\begin{proof} As before, we can assume that $P$ is elementary abelian and the prime $p$ is bounded. Because of Lemma \ref{qbounded} we can assume that the size of the definition field is bounded. Then the order of $G$ is bounded in terms of the rank and so without loss of generality we can assume that $G$ is of classical type. Here $V$ is the natural module for $G$. It is more convenient to work in the linear group so we lift $G$ to a subgroup of ${\mathrm{GL}}(V)$, which we still denote by $G$. Let $T$ be the full preimage of $P$ in ${\mathrm{GL}}(V)$. Then $T'$ has order at most $p$ and so $T$ contains an elementary abelian $p$-subgroup $H$ of order at least $|P|^{\frac 12}$. Taking into account that $p$ is bounded we invoke Lemma \ref{lift}, which shows that there is a positive number $\delta$, depending only on $\ep$, such that for every $g\in G$ we have $\pr(H,H^g)\ge\delta$. As it is enough to bound the order of $H$, we may assume that $P$ is an elementary abelian subgroup of ${\mathrm{GL}}(V)$. It will be convenient to work under the additional assumption that $P\cap Z(G)=1$. This can be done without loss of generality, taking into account that $P$ contains a subgroup of index $p$ that avoids $Z(G)$ and  replacing $P$ by that subgroup.

Aside from the case of ${\mathrm{SL}}(V)$, we have a form and so the notion of nondegenerate and totally singular are well defined.  In the case of ${\mathrm{SL}}(V)$, we consider all subspaces both totally singular and nondegenerate.

Given a subspace $W$ of $V$, we write $G(W)$ for the group induced by the action of the stabilizer of $W$ on $W$. We say that a $P$-submodule $W$ is 
\emph{good}
  if it is a nondegenerate submodule that is minimal subject to the condition that  $G(W)/Z(G(W))$ is almost simple and  the image of $P$ in $G(W)$ is not central. 
  Note that  if  $G(W)/Z(G(W))$ is almost simple, then  the image of $P$ in $G(W)$ is not central whenever it is non-cyclic. 
  Moreover, if $W$ is good, then the image of $P$ in $G(W)$  does not commute with some of its conjugates.

Since $P$ is not contained in $Z(G)$, good $P$-submodules exist. 

We claim that  the dimension of every good $P$-submodule is bounded. 
Observe that irreducible $P$-submodules all have dimension at most $p-1$ and recall that $p$ is bounded. 
 Assume that $P$ is non-cyclic and observe that the group $G(W)$ is  almost simple modulo its centre whenever $W$ is a nondegenerate subspace of dimension at least $5$. Every $P$-submodule of dimension $d$ is contained in a nondegenerate $P$-submodule of dimension at most $2d$. Let $U$ be a good submodule. There are $2$ irreducible $P$-submodules $U_1,U_2\leq U$ such that the image of $P$ in $G(U_1+U_2)$ is non-cyclic. This is contained in a nondegenerate $P$-submodule $U_3\leq U$ of dimension at most twice that of $U_1+U_2$. If $G(U_3)$ is  almost simple modulo its centre, then, because of minimality, $U_3=U$. Otherwise, the dimension of $U_3$ is at most $4$ and $U\cap U_3^\perp$ contains a nondegenerate $P$-submodule $U_4$, which is a sum of at most two irreducible $P$-submodules. Set $U_5=U_3+U_4$.  If $G(U_5)$ is  almost simple modulo its centre, then because of minimality $U_5=U$. Otherwise, the dimension of $U_5$ is exactly $4$ and $U\cap U_5^\perp$ contains a nondegenerate $P$-submodule $U_6$, which is a sum of at most two irreducible $P$-submodules. It is clear that the dimension of $U_5+U_6$ is at least $6$ and $G(U_5+U_6)$ is  almost simple modulo its centre. Because of minimality, $U=U_5+U_6$. Hence, the dimension of $U$ is bounded.

Let $W$ be a maximal $P$-submodule that can be written as an orthogonal sum of good submodules. Let $Q$ be the kernel of the action of $P$ on $W$. It follows that either $G(W^\perp)$ is not  almost simple modulo its centre or $Q$ embeds into $Z(G(W^\perp))$. In either case the order of $Q$ is bounded and we can replace $P$ by a subgroup that avoids $Q$. So without loss of generality we assume that $P$ acts on $W$ faithfully.

Let now $k$ be the minimal number such that there is a faithful $P$-submodule $U$, which can be written as an orthogonal sum
$$
 U=U_1\oplus U_2\oplus\dots\oplus U_k,
$$ where $U_i$ are good submodules. 

Obviously, $|P|$ is bounded in terms of $k$ and so we now need to show that $k$ is $\ep$-bounded.

Let $Q_i$ be the kernel of the action of $P$ on $U_i$ and note that the index $|P:Q_i|$ is bounded. Since the module $U$ is faithful, $\cap_{i\leq k} Q_i=1$. 
For each $j=1,\dots,m$ set 
\[P_j=\bigcap_{1\leq i\leq k, i\neq j}Q_i.\]
Observe that because of minimality of $k$ the subgroup $P_j$ acts nontrivially on $V_i$ if and only if $i=j$. It follows that the product $\prod P_j$ is direct. Possibly some of the factors $P_j$ are central in $G(U_j)$. To avoid such cases we consider larger submodules.

Namely, let $l=\lfloor k/2\rfloor$ and set $R_i=P_{2i-1}P_{2i}$ and $W_i=U_{2i-1}\oplus U_{2i}$ for $i=1,\dots,l$.

Since $R_i$ is not cyclic, for each $i\in\{1,\dots,l\}$ there is $g_i\in G$ stabilizing $W_i$ and trivial on $W_i^{\perp}$ such that $R_i$ does not commute with $R_i^{g_i}$. We remark that the elements $g_i$ commute and $[R_j,g_i]=1$ if and only if $i\ne j$.
 Let $R_0=\prod_iR_i$. Note that $R_0$ is a direct product of the $R_i$.

 On the one hand, we have $\pr ( {R}_0, {R}_0^{ {g}})\ge \pr (P, P^g) \ge \ep$ for every $ {g}\in {G}$.
On the other hand, for $g=g_1\cdots g_l$ we have that $R_j^g=R_j^{g_j}$ and, since $\pr(R_j,R_j^{g_j})\le 3/4$ by Lemma \ref{lem:2.2}(\ref{3/4}), we conclude that 
 \[\ep\le\pr({R}_0,{R}_0^{{g}})\le \prod_{j=1}^l\pr ({R}_j,{R}_j^{{g}_j})\le \left(\frac 34\right)^l.\]
 This implies that $l\le\frac{\log \ep}{\log (3/4)}$ is $\ep$-bounded. Hence, $k$ is bounded and this completes the proof.
\end{proof}

We can now complete the proof of Theorem \ref{almostsimple}.

\begin{proof}[Proof of Theorem \ref{almostsimple}]  Obviously, we have nothing to prove if $S$ is a sporadic group. In the case where $S$ is the alternating group the result is established in Lemma \ref{alter}. So we assume that $S$ is of Lie type over the field of size $q$. As before, we may assume that $P$ is elementary abelian. Observe that $|P:(S\cap P)|$ is at most $p^4$ (see \cite{atlas}) 
  so we replace $P$ by $S\cap P$ and assume that $G$ is simple. Let $q_0$ be as in Lemma \ref{qbounded}. If $q>q_0$, the theorem follows from Lemma \ref{qbounded}. Moreover, if $p$ is coprime to $q$, the theorem follows from Lemma \ref{difchara}. Thus, we are reduced to the case where $q$ is a $p$-power and $q\leq q_0$. In this case $|G|$ is bounded in terms of the rank and $\ep$. The proof is complete.
\end{proof}

\section{Theorems \ref{main1} and \ref{main2}}\label{thms1.1-1.3} 

Here we complete the proofs of Theorems  \ref{main1} and \ref{main2} as well as corollaries mentioned in the introduction.

\begin{proposition}\label{soluble-tot}
Assume Hypothesis \ref{basic}. If $G$ is soluble and $O_p(G)=1$, then the order of $P$ is bounded. 
\end{proposition}
   \begin{proof}
Note that $F(G)$ is a $p'$-group and $C_P(F(G))=1$. We consider the subgroup $F(G)P$ and apply Proposition \ref{soluble}. It follows that $|P|$ is bounded. 
 \end{proof}

Recall that the generalized Fitting subgroup $F^*(G)$ of a finite group $G$ is the subgroup generated by $F(G)$ and all components of $G$, where a component is a subnormal subgroup which is a perfect central extension of a simple group. 

We are ready to prove Theorem \ref{main2}, which we restate here for the reader's convenience.\medskip

{\it Assume Hypothesis \ref{basic} and suppose that the composition factors of $G$ isomorphic to simple groups of Lie type in characteristic $p$ have Lie rank at most $n$. Then the order of $P$ modulo $O_p(G)$ is bounded in terms of $\epsilon$ and $n$ only. }

\begin{proof}[Proof of  Theorem \ref{main2}]   Without loss of generality we may assume that $O_p(G)=1$.
Let  $R(G)$ be the soluble radical of $G$. In view of Proposition \ref{soluble-tot} the intersection $P\cap R(G)$ has bounded order. Therefore we may pass to the quotient $G/R(G)$ and assume that $R(G)=1$. 

Let  $M=F^*(G)$, which  is semisimple, and let $S$ be a simple factor  of $M$. 
By Theorem \ref{almostsimple}, the order of $N_P(S)/C_P(S)$  is bounded in terms of $n$ and $\epsilon$ only.  
Moreover, by Proposition \ref{orbit} also the index of $N_P(S)$ in $P$ is bounded.  
Therefore there exists an $(n,\epsilon)$-bounded constant $C$ such that  $|P:C_P(S)| \le C$ for every simple factor $S$ of $M$. 

Now an application of Proposition \ref{trickB} leads us to the conclusion that the order of $P$ is $(n,\epsilon)$-bounded, as claimed. 
\end{proof}

We will now prove Theorem \ref{main1}.\medskip

{\it Assume Hypothesis \ref{basic} with $P$ being a Sylow $p$-subgroup of $G$. Then the order of $P$ modulo $O_p(G)$ is bounded. Moreover $O_p(G)$ is bounded-by-abelian-by-bounded.}
\begin{proof}[Proof of  Theorem \ref{main1}]  By Neumann's Theorem \ref{neumann} $P$ is bounded-by-abelian-by-bounded. Hence, the claim about  $O_p(G)$ is immediate and we only need to show that $P/O_p(G)$ has bounded order. Pass to the quotient $G/O_p(G)$ and simply assume that $O_p(G)=1$. In view of Lemma \ref{exp-bounded} the exponent of $P$ is bounded. Observe that every classical group of Lie rank $r$ contains the alternating group $A_r$ of degree $r$. Since the exponent of the Sylow $p$-subgroup of $A_r$ goes to infinity as $r$ does, it follows that the composition factors of $G$ that are simple groups of Lie type in characteristic $p$ have bounded Lie ranks. Thus, by Theorem \ref{main2}, $P/O_p(G)$ has bounded order.
\end{proof}

We will now describe the proofs of the corollaries mentioned in the introduction. We start with Corollary \ref{cor1}:

{\it If $G$ is a finite group such that $\pr(P_1,P_2)\geq\ep>0$ for every $p\in\pi(G)$ and every two Sylow $p$-subgroups $P_1,P_2$ of $G$, then $G$ is bounded-by-abelian-by-bounded.}

\begin{proof}[Proof of Corollary \ref{cor1}] Theorem \ref{main1} tells us that, for every $p\in\pi(G)$, a Sylow $p$-subgroup of $G/F(G)$ has bounded order. Since the bound does not depend on the prime $p$, it follows that 
 only boundedly many primes divide the order of $G/F(G)$ and so $|G/F(G)|$ is bounded. Therefore,
 we can assume that $G$ is nilpotent. In view of Theorem \ref{neumann} every Sylow subgroup of $G$ is bounded-by-abelian-by-bounded. Again, since the bounds do not depend on primes, it follows that only boundedly many Sylow subgroups of $G$ are nonabelian. Hence the result.
\end{proof}

Next, we proceed with Corollary \ref{cor3}. \medskip

{\it Let a finite group $G$ contain a subgroup $K$ such that $\pr(K,K^x)\geq\ep>0$ for every $x\in G$.
\begin{enumerate} 
\item If $K$ is a Hall subgroup of $G$, then the order of $K$ modulo the Fitting subgroup $F(G)$ is bounded in terms of $\ep$ only.
\item If $n$ is the maximum of Lie ranks of composition factors of $G$, then the order of $K$ modulo $F(G)$ is bounded in terms of $n$ and $\ep$ only. 
\end{enumerate}
}
\begin{proof}[Proof of Corollary \ref{cor3}] If $K$ is a Hall subgroup of $G$, then every Sylow subgroup of $K$ is also a Sylow subgroup of $G$. Theorem \ref{main1} states that the order of every Sylow subgroup of $K$ modulo $F(G)$ is bounded. Hence, the order of $K$ modulo $F(G)$ is bounded as well.

If $n$ is the maximum of Lie ranks of composition factors of $G$, then by Theorem \ref{main2} the order of every Sylow subgroup of $K$ modulo $F(G)$ is bounded in terms of $n$ and $\ep$ only.
Thus, the order of $K$ modulo the Fitting subgroup $F(G)$ is bounded in terms of $n$ and $\ep$ only. 
\end{proof}

\section{Proof of Theorem \ref{main3}}\label{thm3}

The upper insoluble series  
\[1\leq R_1\leq L_1\leq R_2\leq L_2\leq\dots\] of a finite group $G$ is recursively defined as follows. Let  $R_1=R_1(G)$ be the soluble radical of $G$ and $L_1=L_1(G)$ the full inverse image of $F^*(G/R_1)$. Then $R_k=R_k(G)$ is the full inverse image of the soluble radical of $G/L_{k-1}$ and $L_k$ is the full inverse image of $F^*(G/R_k)$.
 
\begin{proposition}\label{insoluble} Assume Hypotheses \ref{basic}. Then the order of $P$ modulo $R_2(G)$ is $\epsilon$-bounded.
\end{proposition}
\begin{proof} Without loss of generality we assume that $R_1(G)=1$ and set $L=L_1(G)$. Then $L$ is semisimple and $P$ acts on $L$ by conjugation. The subgroup $P$ permutes the simple factors of $L$. Let $P_0$ be the kernel of the permutational representation of $P$ on the set of simple factors of $L$. Then $P_0\leq R_2$ (see for example \cite[Lemma 2.1]{KhSh15}) 
 and so we need to show that $P_0$ has bounded index in $P$.

Suppose there are exactly $k$ $P$-orbits under the permutational action of $P$ on the simple factors of $L$ and let $M_i$ be the product of the simple factors in the $i$th orbit. 
%In view of Proposition \ref{orbit} each $M_i$ is a product of boundedly many, say at most $l$ simple factors. 
 For $i=1,\dots,k$  choose a simple factor $S_i\leq M_i$.
In each simple factor $S_i$ choose elements $a_i$ and $b_i$ such that $S_i$ is generated by $a_i$ and $b_i$. Let 
\[ a=\prod_{1\leq i\leq k}a_i  \quad \text{ and } \quad b=\prod_{1\leq i\leq k}b_i.\]

By Lemma \ref{lemQ} there exists a normal subgroup $Q_{a}$ of bounded index in $P$ with the property that, for every $g\in Q_{a}$, there is a normal subgroup $D_{a,g}$ of bounded index in $P$ such that $g$ centralizes  $[a,D_{a,g}]$. Similarly, there exists normal subgroups $Q_b$ and $D_{b,g}$ of bounded index in $P$ with the analogous properties. Set $Q=Q_{a} \cap Q_{b}$ and observe that $Q$ is a normal subgroup of bounded index in $P$ with the property that for every element $g\in Q$ there is a normal subgroup $D=D_{a,g} \cap D_{b,g} $ of bounded index in $P$ such that $g$ centralizes both $[a,D]$ and $[b,D]$. 

Assume that  there is $g\in Q$ such that $S_{i}^g\neq S_{i}$ for an index $i$. Since $g$ centralizes $[a,D]$ and $[b,D]$, and since the factors $S_1,S_2,\dots,S_k$ belong to different orbits under the action of $P$,  for every $d\in D$ and every index $i$, the element $g$ centralizes $[a_{i},d]$ and $[b_{i},d]$. 

Keeping in mind that $S_i$ is generated by $a_i$ and $b_i$, observe that if $[a_i,D]=1$ and $[b_i,D]=1$, then $[S_i,D]=1$. In this case the index of $C_P(S_i)$ in $P$ is at most some bounded constant, say $t_0$. 
 
Otherwise, there exists an element $d \in D$ such that either $[a_i,d]\neq1$ or $[b_i,d]\neq1$. Then $g$ has nontrivial centralizer in $S_iS_{i}^d$. In this case $\{S_{i},S_{i}^d\}$ is a $g$-orbit and  $S_{i}^d= S_{i}^g$. 
Note that $S_{i}^d\neq S_{i}$ because  $S_{i}^g\ne S_{i}$. 

If $p$ is odd, then $g$ cannot have an orbit of size $2$. It follows that either $Q$ normalizes $S_i$ or the index of $C_P(S_i)$ in $P$ is at most the constant $t_0$. 

Now suppose $p=2$ and let $x\in\{a_i,b_i\}$ be such that $[x,d]\neq1$.
 The fact that $g$ centralizes  $[x,d]$ means that $g^2$ centralizes $x$ and 
 \[(x^{-1})^d=x^g.\]
  The right-hand side of the above equality does not depend on $d$. It follows that the index of $C_D(x)$ in $D$ is at most $2$. We conclude that the index of $C_D(S_i)$ in $D$ is at most $4$. So again we obtain that the index of $C_P(S_i)$ in $P$ is at most a constant, which still will be denoted by $t_0$. 

Thus, in all cases either $Q$ normalizes $S_i$ or the index of $C_P(S_i)$ in $P$ is at most $t_0$. Without loss of generality assume that the index of $C_P(S_i)$ in $P$ is at most $t_0$ for $i=1,\dots,m$ and $Q$ normalizes $S_i$ for $i=m+1,\dots,k$.

Since $Q$ is normal in $P$, it follows that $Q$ normalizes all simple factors in $M_i$ for $i=m+1,\dots,k$. 
Moreover, for $i=1,\dots,m$ and every simple factor $S$ of $M_i$ we have that $|P:C_P(S)| \le t_0$.
%Recall that each $M_i$ is a product of at most $l$ simple factors, where $l$ is bounded. Therefore for $i=1,\dots,m$ the index of $C_P(M_i)$ in $P$ is at most $t_0^l$.
 Set $N=\prod_{1\leq i\leq m}M_i$. According to Proposition \ref{trickB} $C_P(N)=O_p(NP) \cap P$ has bounded index in $P$. Obviously $Q\cap C_P(N)\leq P_0$ and so indeed $P_0$ has bounded index in $P$, as required.
\end{proof}

In what follows we write $H^{(\infty)}$ for the last term of the derived series of a finite group $H$.
 
\begin{lemma}\label{layer} Let $N$ be a perfect normal subgroup of a finite group $H$ such that $N/Z(N)$ is semisimple. Then $N\leq E(H)$.
\end{lemma}
\begin{proof} Write $N/Z(N)=S_1\times\dots\times S_n$, where $S_i$ are simple. Let $N_i$ denote the full preimage of $S_i$ in $N$, and set $L_i=N_i^{(\infty)}$. Observe that $L_i$ is a subnormal quasisimple subgroup of $H$. Set $L=\prod L_i$. It is immediate that $L\leq E(H)$. Moreover, $N/L$ is abelian. Since $N$ is perfect, we conclude that $N=L$ and so $N\leq E(H)$.
\end{proof}

Recall that $E_p(H)$ denotes the product of the components of order divisible by $p$ of a finite group $H$. We write $O_{p,e_p}(H)$ for the full inverse image of $E_p(H/O_p(H))$ and $O_{p,e_p,p}(H)$ for the full inverse image of $O_p(H/O_{p,e_p}(H))$.

We will now prove Theorem \ref{main3}. 
 
 \begin{proof}[Proof of Theorem \ref{main3}]
Recall that $G$ is a finite group containing a $p$-subgroup $P$ such that $\pr(P,P^x)\geq\ep>0$ for every $x\in G$. We want to prove that the order of $P$ modulo $O_{p,e_p,p}(G)$ is bounded in terms of $\ep$ only. We may assume that $O_p(G)=1$. Write $R=R_2(G)$. In view of Proposition \ref{insoluble}, $P\cap R$ has bounded index in $P$. We replace $P$ by $P\cap R$ and without loss of generality assume that $P\leq R$. Let $T$ be the $p$-soluble radical of $R$, i.e. the largest normal subgroup of $R$ whose nonabelian simple composition factors have order coprime to $p$, and let $E/T=E_p(R/T)$. Note that $E=L_1(G)$ and so $R/E$ is soluble. Taking into account Corollary \ref{cor2}, 
it follows that the order of $P$ modulo $O_p(TP)$ is bounded. By Proposition \ref{soluble-tot}, the same is true for the order of $EP/E$ modulo $O_p(R/E)$. So replacing $P$ by a subgroup of bounded index we reduce to the case where $P\leq O_p(TP)$ and $EP/E\leq O_p(R/E)$. It follows that $[T,P]$ is a $p$-subgroup normalized by $T$. Since $[T,P]$ is subnormal in $G$ and since $O_p(G)=1$, we deduce that $[T,P]=1$.

Let $Y=\langle P^G\rangle$. As $EP/E\leq O_p(R/E),$ the group $Y/(Y\cap E)$ is a $p$-group. Note that $T\cap Y\leq Z(Y)$, whence it follows that the $p$-soluble radical of $Y$ is central. The group $(Y\cap E)/(Y\cap T)=(Y\cap E)/Z(Y\cap E)$ is semisimple since it is isomorphic to a normal subgroup of $E/T$. Let $Y_1=Y^{(\infty)}$. Observe that $Y_1$ is perfect and $Y_1/Z(Y_1)$ is semisimple. By virtue of Lemma \ref{layer}, $Y_1=E_p(Y)$. Moreover, $Y/E_p(Y)Z(Y)$ is a $p$-group.
So $Y=Z(Y)E_p(Y)S$ for a  Sylow $p$-subgroup  $S$ of $Y$. It follows that $Y/E_p(Y)$ is nilpotent. 
Since $Y$ is generated by $p$-elements, we deduce that $Y/E_p(Y)$ has $p$-power order and so $Y=E_p(Y)S$. Therefore $Y\leq O_{e_p,p}(G)$, as required.
\end{proof}

\section{Theorem \ref{main4}}\label{thm4} 

Recall  that a profinite group is a compact Hausdorff topological space, where a basis of the neighbourhoods of the identity consists of normal subgroups of finite index. 
 
The Borel $\sigma$-algebra $\mathcal M$ of a profinite group $G$ is the one generated by all closed subsets of $G$ and there is a unique Haar measure $\mu$ on $(G,\mathcal M)$ such that $\mu(G) = 1$. 

Moreover, if $H$ is a subgroup of $G$, then either $\mu(H) =0$ or $\mu(H) >0$;  in the latter case $H$ is open in $G$ and $\mu(H)=|G:H|^{-1}.$

If $H$ and $K$ are  subgroups of $G$, the set 
\[C = \{(x, y) \in H \times K \mid xy = yx\}\]  is closed in $H \times K$, since it is the  preimage of $1$ under the continuous map $f : H\times K \rightarrow G$ given by $f(x, y) = [x, y]$.
Denoting the normalized Haar measures of $H$ and $K$ by $\mu_H$ and $\mu_K$, respectively, the probability that a random element from $H$ commutes with a random element from $K$ is defined as 
 \[\pr(H,K) = (\mu_H \times \mu_K)(C).\]
 
Note that $\pr(H_0,K) \ge \pr(H,K)$ for every subgroup $H_0$ of $H$, and 
 $\pr(HN/N , KN/N) \ge \pr(H,K)$  for every normal subgroup $N$ of $G$. 
 (see \cite[Lemmas 2.3, 2.4]{DMS-Syl_prof}).

 \begin{lemma}\label{limit}
Let $G$ be a profinite group and $H, K \le G$.  Then 
\[ \pr(H,K) = \inf_{N \lhd_o G} \pr \left(\frac{HN}{N},\frac{KN}{N}\right).\]
\end{lemma}
 \begin{proof} 
Let 
\[ C(H,K)=\{ (x,y) \in H\times K \mid [x,y]=1\}. \] 
By definition,  
$\pr(H,K)= (\mu_H \times \mu_K)  \left( C(H,K) \right).$
For an open normal subgroup $N$ of $G$,  we set $H_N=H/(H\cap N)$ and  $K_N=K/(K\cap N)$ 
 and define 
 \begin{eqnarray*}
\bar C_N( H_N,  K_N)
 = \left\{ (x(H\cap N),y(K\cap N)) \in H_N\times K_N \mid   [x,y]\in N \right\}.
\end{eqnarray*}
 Consider the  epimorphism  $\pi_N$ from $H \times K$ to $H_N \times K_N$ 
and let $\Omega_N( H,  K)$  be the  full preimage of  $ \bar C_N( H_N,  K_N)$, that is 
\begin{eqnarray*}
\Omega_N( H,  K) 
 = \{ (x,y) \in H\times K \mid   [x,y]\in N\}. 
\end{eqnarray*}
By the properties of the Haar measure  (see e.g. \cite[Proposition 18.2.2]{Jarden}) 
 we have that
\begin{eqnarray*}
 (\mu_H \times \mu_K ) \left( \Omega_N( H, K) \right) &=&
( \mu_{H_N} \times \mu_{K_N})  \left( \bar C_N( H_N,  K_N) \right)
\end{eqnarray*}
 Note that  $(x(H\cap N),y(K\cap N)) \in H_N\times K_N$ satisfies $ [x,y]\in N$ if and only if 
  its image $(xN,yN) $ under the isomorphisms from $H_N \times  K_N $ to $ {HN}/{N} \times  {KN}/{N}$  satisfies $  [xN,yN]=1$. 
  This implies that
 \[ \big| \bar C_N( H_N,  K_N)\big| =\left| C\left(\frac{HN}{N}, \frac{KN}{N}\right)\right| ,\] 
 therefore
 \begin{eqnarray*}
( \mu_{H_N} \times \mu_{K_N})  \left( \bar C_N( H_N,  K_N) \right)
&=&\left( \mu_{\frac{HN}{N}} \times \mu_{\frac{KN}{N}}\right)  \left(C \left(\frac{HN}{N}, \frac{KN}{N}\right) \right) \\
 &=&\pr \left(\frac{HN}{N},\frac{KN}{N}\right).
\end{eqnarray*}
It follows that 
\[ (\mu_H \times \mu_K ) \left( \Omega_N( H, K) \right) = \pr \left(\frac{HN}{N},\frac{KN}{N}\right).\]

Note that if $N_1 \le N_2$ then 
\[ \Omega_{N_1} ( H ,  K ) \le \Omega_{N_2} ( H,  K )\]
 so we have a descending  chain of subgroups $\{ \Omega_{N}( H, K)  \}_{N  \lhd_o G}$ and thus 
 \[  (\mu_H \times \mu_K ) \left( \bigcap_{N  \lhd_o G}  \Omega_{N}( H, K)   \right) =
  \inf_{N  \lhd_o G}  (\mu_H \times \mu_K ) \left(   \Omega_{N}( H, K)   \right). \]
 Since $  (\mu_H \times \mu_K ) \left(   \Omega_{N}( H, K)   \right)= \pr ( HN/N, KN/N)$  for every $N$, now it is sufficient to prove that 
 \[  C(H,K)= \bigcap_{N  \lhd_o G}  \Omega_{N}( H, K) .  \]
 It is clear that 
 \[C(H,K) \le  \Omega_{N}( H, K)=\{ (x,y) \in H\times K \mid [x,y] \in N\} \]
  for every $N$ open in $G$. Conversely, 
 consider an element $ (x,y) \in H\times K$ such that $[x,y] \in N$ for every $N \lhd_o G$. Then 
 \[ [x,y] \in \bigcap_{N  \lhd_o G} N =1, \] 
 hence $(x,y) \in C(H,K) $. 
 \end{proof}

\begin{proposition}\label{profinite-easy} Let $G$ be a profinite group and $p\in\pi(G)$. Assume that $\pr(P_1,P_2)\geq\ep>0$ for every two Sylow $p$-subgroups $P_1,P_2$ of $G$. Then a Sylow $p$-subgroup of $G/O_p(G)$   is finite of order bounded in terms of $\ep$ only. 
%Moreover $O_p(G)$ is bounded-by-abelian-by-bounded.
\end{proposition}
\begin{proof} 
It   follows from  Theorem \ref{main1} that there exists an $\ep$-bounded  integer $l$ such that,  
 for every open normal subgroup $N$ of $G$, the order of $PN/N$ over $O_p(G/N)$ is at most $l$. Thus, the routine inverse limit argument yields the result. \end{proof}

\begin{lemma}\label{p-semisimple} 
Let $G$ be a profinite group containing a Sylow $p$-subgroup $P$ such that $\pr(P,P^x)>0$ for every $x\in G$. If $G$ is topologically isomorphic to a Cartesian product of nonabelian finite simple groups of order divisible by $p$, then $G$ is finite.
\end{lemma}
\begin{proof} 
Let $G=\prod_{i\in I}S_i$,  where $p$ divides $|S_i|$, for every $i$. 
For every $i\in I$, fix a Sylow $p$-subgroup $R_i$ of $S_i$ and choose $g_i\in S_i$ such that $R_i\neq R_i^{g_i}$.
  It follows from   Lemma \ref{lem:2.2}(\ref{3/4}) 
  %\ref{lem:2.7} 
  % \cite[Lemma 2.7]{DLMS-finite}  
  that 
  \[ \pr(R_i, R_i^{g_i}) 
  %\le \frac{2p-1}{p^2}
   \le \frac{3}{4}. \] 
  Taking into account   Lemma \ref{lem:2.2}, 
   %\cite[Lemma 2.2]{DLMS-finite},
    for every finite subset $I_0$ of $I$ we have 
    \[  \pr \left(\prod_{i\in I_0} R_i, \prod_{i\in I_0}   R_i^{g_i}\right) = \prod_{i \in I_0} \pr(R_i, R_i^{g_i}) \le \left( \frac{3}{4} \right)^{|I_0|},\]
      On the other hand,   
    $\prod_{i \in I} R_i$ and $\prod_{i \in I}  R_i^{g_i}$ are  
    two Sylow $p$-subgroups of $G$, whence 
   \[  \pr \left(\prod_{i\in I_0} R_i, \prod_{i\in I_0}   R_i^{g_i}\right)  \ge  \pr \left(\prod_{i\in I} R_i, \prod_{i\in I}   R_i^{g_i}\right)  \ge \epsilon \] 
 for some $\epsilon >0$. 
 So the cardinality of $I_0$ is $\epsilon$-bounded. 
  We conclude that $I$ is finite and the lemma follows.  
 \end{proof}

It is well known that every finite group $K$ possesses a series
\[1 = K_0 \le K_1 \le\dots \le K_{2h+1} = K\]
of normal subgroups such that $K_{i+1}/K_i$ is $p$-soluble (possibly trivial) if $i$ is
even and a direct product of nonabelian simple groups of order divisible by $p$ if $i$ is odd. Following \cite{KhSh15} the number
of  non-$p$-soluble factors in this series is called the non-$p$-soluble length $\lambda_p(K)$ of $K$.  Wilson showed in \cite{Wilson} that if $\mathcal{X}_1,\dots, \mathcal{X}_n$ are classes of finite groups closed with  respect to normal subgroups and subdirect products and 
if $\mathcal{X}$ is the class of finite groups having a normal series of given length $n$ such that
the $i$-th section is a $\mathcal{X}_i$-group, then every pro-$\mathcal{X}$ group has a normal series of length $n$  such that
the $i$-th section is a pro-$\mathcal{X}_i$-group. In particular, a combination of Lemma 2 and Lemma 3 of
\cite{Wilson} shows that if $\mathcal{X}$ is the class of finite groups $K$ such that $\lambda_p(K) \le l,$ then
every pro-$\mathcal{X}$ group has a normal series of  length at most $2l+1$ each of whose factors is
either pro-$p$-soluble or a Cartesian product of nonabelian finite simple groups of order divisible by $p$.

Let  $G$ be a finite $p$-soluble group. Recall that the $p$-length $l_{p}(G)$ of $G$ is the minimum number of $p$-factors in a $(p, p')$-series of $G$. 
   It is a well-known corollary of the Hall-Higman theory that $l_{p}(G)$ is bounded in terms of the derived length of a Sylow $p$-subgroup of $G$ (see \cite{hh, bryukhanova}). As above, it follows from \cite{Wilson}  that if $\mathcal{X}$ is the class of finite $p$-soluble groups $K$ such that $l_p(K)\le l,$ then
every pro-$\mathcal{X}$ group has a normal series of  length at most $2l+1$ each of whose factors is
either a  pro-$p$ group or a pro-$p'$ group.

We recall that $O_{p,p'}(G)$ is the full preimage in $G$ of the largest normal $p'$-subgroup $O_{p'}(G/O_p(G))$ of the group $G/O_p(G)$.

We are ready to prove Theorem \ref{main4}, which we restate here for the reader's convenience.\medskip

{\it
Let $G$ be a profinite group and $p\in\pi(G)$. Assume that $\pr(P_1,P_2)>0$ for every two Sylow $p$-subgroups $P_1,P_2$ of $G$. Then $O_{p,p'}(G)$ is open. Moreover $O_p(G)$ is virtually abelian.
}
\begin{proof}[Proof of Theorem \ref{main4}]
Since $ \pr(P, P)=\epsilon >0$, a profinite version of  Neumann's theorem (see \cite{LP}) shows that  $P$ is  virtually abelian and so is  $O_p(G)$.

Moreover  $P$ is soluble, say of derived length $d$.  
 It follows from \cite[Theorem 1.4]{KhSh15} that the non-$p$-soluble length of every finite quotient of $G$ is $d$-bounded.
Therefore $G$ has a normal series of finite length 
\[1\leq N_1\leq\dots\leq N_s=G\]
 in which every section is either pro-$p$-soluble or a Cartesian product of nonabelian finite simple groups of order divisible by $p$.
     The  sections $N_{i+1}/N_i$ of the last type are finite by Lemma \ref{p-semisimple}, 
 so there is an open normal subgroup $H$ of $G$  that $H \cap N_{i+1}  \le N_i$ for every such section, hence $H$ is pro-$p$-soluble. 
 Thus, we deduce that $G$ is virtually pro-$p$-soluble. Since  $O_{p,p'}(H) \le O_{p,p'}(G)$, we now may assume that $G$ is pro-$p$-soluble. 

As $P$ is soluble of derived length $d$, by the Hall-Higman theory the $p$-length of every finite quotient of $G$  is $d$-bounded. It follows that $G$ possesses a normal series of finite length
\[1 = K_0 \le K_1 \le\dots \le K_h=G,\]
where every section is either a pro-$p$ or a pro-$p'$ group. Without loss of generality we may assume that $O_p(G)=1$. Now it is enough to show that $P$ is finite, because in this case  $G$ has an open normal pro-$p'$ subgroup.

Since $O_p(G)=1$, the subgroup $P$  acts faithfully on $M=O_{p'}(G)$ by conjugation.
Therefore we can assume that $G=MP$.

Since 
\[ \pr(P, P^g) >0, \ \textrm{ for every } \ g \in M, \]
the subgroup $M$ is the union of  the sets
  \[ X_n=\{g \in M \mid \pr(P, P^g) \ge 1/n\}. \]
Note that the set $X_n$ is closed for every $n \in \N$. Indeed, 
 consider an element $a \notin X_n$; then $\pr(P, P^a) < 1/n$. It follows from Lemma \ref{limit}
  that there is an open normal subgroup $N$ of $G$ such that $ \pr ( PN/N, P^{a}N/N)  < 1/n$. Observe that for all $x \in N \cap M$, 
   \[ \pr(P, P^{ax}) \le  \pr (PN/N, P^{a}N/N) < 1/n \] 
   which shows that $ax \notin X_n$. Therefore, $M \setminus X_n$ is open and so $X_n$ is closed. 
  
By the Baire category theorem (see \cite[p. 200]{Ke}) at least one of the sets $X_n$ contains a non-empty interior. Hence, there is a $P$-invariant open normal subgroup $N$ of $M$ and an element $a\in M$ such that $\pr(P, P^{ax})\ge 1/n$ for every $x \in N$. 

 Since $M/N$ is finite and the group $P$ naturally acts on $M/N$ by conjugation, the subgroup $Q=C_{P}(M/N)$ is open in $P$ and
 \[ [M, Q] \le N. \] 

 It follows from Lemma \ref{cc} and the standard inverse limit argument that
\[ M=N C_M(Q). \]
 So we can assume that $a \in C_M(Q)$. 
Then, if $g \in aN$, we have $Q^g=Q^y$ for some $y \in N$. 
As $\pr(Q,Q^x)\ge \pr(P,P^x)$ for every $x\in G$, it follows from the choice of $N$ that there is $\epsilon\ge \frac 1n$ such that $\pr(Q,Q^x)\ge\ep$ for every $x\in QN$.
 Thus, we deduce from Proposition \ref{profinite-easy}  that $Q/O_p(QN)$ is finite. As $O_p(QN)$ centralizes $N$ and it is contained in the Sylow $p$-subgroup $Q$ of $QN$,  it follows that $C_Q(N)$ is open in $Q$. Since $M=NC_M(Q)$, we deduce that $C_Q(M)$ is open in $Q$ and also in $P$. As $C_Q(M)\le C_P(M)\le O_p(G)=1$, the subgroup $P$ has finite order.
 \end{proof}

\end{document}